\documentclass[12pt,reqno]{amsart}

\usepackage[T1]{fontenc}
\usepackage[utf8]{inputenc}
\usepackage{lmodern}

\usepackage[a4paper,margin=1in]{geometry}
\usepackage{setspace}
\allowdisplaybreaks

\usepackage{amsmath,amssymb,amsthm,mathtools,amsfonts,bm}
\usepackage{mathrsfs}
\usepackage{esint}
\numberwithin{equation}{section}
\usepackage{graphicx}
\usepackage{subcaption}
\usepackage{booktabs}
\usepackage{multirow}
\usepackage{tabularx}
\usepackage{colortbl}
\usepackage{float}
\usepackage{comment}

\usepackage{tikz}
\usetikzlibrary{arrows.meta,calc,positioning}
\usepackage{pgfplots}
\pgfplotsset{compat=1.18}

\usepackage{enumitem}
\usepackage{lscape}

\usepackage{xcolor}
\usepackage{hyperref}
\hypersetup{
  colorlinks=true,
  linkcolor=blue!50!black,
  citecolor=blue!50!black,
  urlcolor=blue!50!black,
  pdfauthor={Behrooz Moosavi Ramezanzadeh},
  pdftitle={An asymptotic mean value characterization for the regularized p-Laplacian}
}
\usepackage[nameinlink,capitalise,noabbrev]{cleveref}

\usepackage{aliascnt}

\theoremstyle{plain}
\newtheorem{theorem}{Theorem}[section]

\newaliascnt{proposition}{theorem}
\newtheorem{proposition}[proposition]{Proposition}
\aliascntresetthe{proposition}

\newaliascnt{lemma}{theorem}
\newtheorem{lemma}[lemma]{Lemma}
\aliascntresetthe{lemma}

\newaliascnt{corollary}{theorem}
\newtheorem{corollary}[corollary]{Corollary}
\aliascntresetthe{corollary}

\theoremstyle{definition}

\newaliascnt{definition}{theorem}

\aliascntresetthe{definition}

\newaliascnt{example}{theorem}
\newtheorem{example}[example]{Example}
\aliascntresetthe{example}

\theoremstyle{remark}
\newaliascnt{remark}{theorem}
\newtheorem{remark}[remark]{Remark}
\aliascntresetthe{remark}

\newaliascnt{observation}{theorem}
\newtheorem{observation}[observation]{Observation}
\aliascntresetthe{observation}

\newaliascnt{claim}{theorem}

\aliascntresetthe{claim}

\crefname{theorem}{Theorem}{Theorems}
\Crefname{theorem}{Theorem}{Theorems}
\crefname{proposition}{Proposition}{Propositions}
\Crefname{proposition}{Proposition}{Propositions}
\crefname{lemma}{Lemma}{Lemmas}
\Crefname{lemma}{Lemma}{Lemmas}
\crefname{corollary}{Corollary}{Corollaries}
\Crefname{corollary}{Corollary}{Corollaries}
\crefname{definition}{Definition}{Definitions}
\Crefname{definition}{Definition}{Definitions}
\crefname{example}{Example}{Examples}
\Crefname{example}{Example}{Examples}
\crefname{remark}{Remark}{Remarks}
\Crefname{remark}{Remark}{Remarks}
\crefname{observation}{Observation}{Observations}
\Crefname{observation}{Observation}{Observations}
\crefname{claim}{Claim}{Claims}
\Crefname{claim}{Claim}{Claims}
\crefname{assumption}{Assumption}{Assumptions}
\Crefname{assumption}{Assumption}{Assumptions}

\DeclareMathOperator{\divgg}{div}

\DeclareMathOperator{\dist}{dist}

\newcommand{\R}{\mathbb{R}}

\newcommand{\talpha}{\widetilde\alpha}
\newcommand{\tbeta}{\widetilde\beta}
\newcommand{\Sp}{S^+}
\newcommand{\Sm}{S^-}
\newcommand{\rhoeps}{\rho_\varepsilon}
\newcommand{\DeltaiR}{\widetilde\Delta_\infty}
\newcommand{\Reg}{\mathcal{F}}
\newcommand{\Lp}{\mathcal L_p}

\title[Asymptotic mean value for the regularized
\texorpdfstring{$p$}{p}-Laplacian]%
{An asymptotic mean value characterization for the regularized
\texorpdfstring{$p$}{p}-Laplacian}

\author{Behrooz Moosavi Ramezanzadeh}

\thanks{University of Pittsburgh, Pittsburgh, PA, USA.
\texttt{behroozmoosavi@pitt.edu}}

\date{\today}

\begin{document}

\begin{abstract}
We characterize solutions of the regularized $p$-Laplace equation
\[
  \divgg\bigl((1+|Dv|^2)^{p/2-1}Dv\bigr)=0,
  \qquad 1<p<\infty,
\]
in an open set $\Omega\subset\R^n$ by a pointwise asymptotic mean value
identity. For $v\in C^2(\Omega)$, solving the equation is
equivalent to
\[
  v(x)
  =
  \frac{\talpha}{2}
  \bigl(\Sp[v](x)+\Sm[v](x)\bigr)
  +
  \tbeta
  \int_{B_\varepsilon(0)}
  v(x+h)\rhoeps(h)\,dh
  +
  o(\varepsilon^2),
\]
where
\[
  \talpha=\frac{p-2}{p+n+1},
  \qquad
  \tbeta=\frac{n+3}{p+n+1}.
\]
Here $\rhoeps$ is the semicircular marginal of normalized Lebesgue measure
on the $(n+1)$-dimensional ball, and $\Sp,\Sm$ are the tilted extremal
functionals arising from the affine lift $w(x,s)=v(x)+s$. The lifted
gradient $(Dv,1)$ never vanishes, so the extremal second-order expansion
is valid at every gradient regime. The characterization holds for the full
range $1<p<\infty$. By standard interior regularity for nondegenerate
regularized $p$-growth equations, weak solutions are smooth in
the interior; the weak and viscosity viewpoints for related quasilinear
$p$-Laplace equations are connected in \cite{JLM01}. The convergence of
the associated projected dynamic programming scheme is established in the
companion paper \cite{Moosavi26}.
\end{abstract}

\maketitle

\noindent\textbf{Acknowledgement.}
I am deeply indebted to Professor Juan Manfredi for proposing this
problem, and for his guidance and valuable comments.

\section{Introduction}
\label{sec:intro}

The classical mean value property states that a function
$u\in C(\Omega)$, where $\Omega\subset\R^n$ is open, is harmonic if and
only if
\[
  u(x)=\fint_{B_\varepsilon(x)}u(y)\,dy
\]
for every $x\in\Omega$ and every
$\varepsilon<\dist(x,\partial\Omega)$. For the standard $p$-Laplacian,
$1<p\le\infty$, Manfredi, Parviainen, and Rossi \cite{MPR10} proved the
asymptotic mean value formula
\[
  u(x)
  =
  \frac{\alpha}{2}
  \Bigl(
    \max_{\overline{B_\varepsilon(x)}}u
    +
    \min_{\overline{B_\varepsilon(x)}}u
  \Bigr)
  +
  \beta\fint_{B_\varepsilon(x)}u(y)\,dy
  +
  o(\varepsilon^2),
  \qquad
  \alpha=\frac{p-2}{p+n},
  \quad
  \beta=\frac{n+2}{p+n}.
\]
This formula is interpreted through smooth test functions, since
$p$-harmonic functions need not admit a pointwise second-order Taylor
expansion. Related discrete notions of $(p,\varepsilon)$-harmonious
functions were developed in \cite{MPR12,LPS14}. The equivalence between
weak and viscosity formulations for related quasilinear $p$-Laplace
equations is treated in \cite{JLM01}.

This paper concerns the regularized equation
\begin{equation}\label{eq:reg-pLap}
  \divgg\bigl((1+|Dv|^2)^{p/2-1}Dv\bigr)=0
  \qquad\text{in }\Omega,
  \qquad 1<p<\infty.
\end{equation}
For $v\in C^2(\Omega)$, a direct computation gives
\begin{equation}\label{eq:decomp-intro}
  \divgg\bigl((1+|Dv|^2)^{p/2-1}Dv\bigr)
  =
  (1+|Dv|^2)^{p/2-1}
  \left(
    \Delta v
    +
    (p-2)
    \frac{\langle D^2v\,Dv,Dv\rangle}{1+|Dv|^2}
  \right).
\end{equation}
Thus the normalized equation is
\[
  \Delta v
  +
  (p-2)
  \frac{\langle D^2v\,Dv,Dv\rangle}{1+|Dv|^2}
  =
  0.
\]
The coefficient matrix
\[
  I+(p-2)\frac{\xi\otimes\xi}{1+|\xi|^2}
\]
has eigenvalues in
\[
  [\min(1,p-1),\max(1,p-1)],
\]
uniformly in $\xi$. Hence the normalized equation is uniformly elliptic
for every finite $p>1$. By classical regularity theory for nondegenerate
quasilinear equations with regularized $p$-growth
\cite{Lieberman88,GilbargTrudinger}, continuous weak solutions are smooth
in the interior.

The denominator $1+|Dv|^2$ is generated by the affine lift
\[
  w(x,s)=v(x)+s.
\]
Indeed,
\[
  Dw=(Dv,1),
  \qquad
  |Dw|^2=1+|Dv|^2.
\]
Projecting the $(n+1)$-dimensional ball back to $\R^n$ changes both the
averaging term and the extremal terms. The average becomes an integral
with respect to
\[
  \rhoeps(h)
  :=
  \frac{2\sqrt{\varepsilon^2-|h|^2}}{|B_\varepsilon^{n+1}|},
  \qquad
  h\in B_\varepsilon(0)\subset\R^n,
\]
the marginal density of normalized Lebesgue measure on
$B_\varepsilon^{n+1}(0)$. The maximum and minimum become
\[
  \Sp[u](x)
  :=
  \sup_{|h|\le\varepsilon}
  \left(
    u(x+h)+\sqrt{\varepsilon^2-|h|^2}
  \right),
\]
and
\[
  \Sm[u](x)
  :=
  \inf_{|h|\le\varepsilon}
  \left(
    u(x+h)-\sqrt{\varepsilon^2-|h|^2}
  \right).
\]
Equivalently, $\Sp[u](x)$ and $\Sm[u](x)$ are the maximum and minimum of
\[
  (h,t)\longmapsto u(x+h)+t
\]
over the ball $B_\varepsilon^{n+1}(0)$.

The main result states that, for every $1<p<\infty$, a function
$v\in C(\Omega)$ solves \eqref{eq:reg-pLap} if and only if, for every
$x\in\Omega$,
\begin{equation}\label{eq:intro-amv}
  v(x)
  =
  \frac{\talpha}{2}
  \bigl(\Sp[v](x)+\Sm[v](x)\bigr)
  +
  \tbeta
  \int_{B_\varepsilon(0)}
  v(x+h)\rhoeps(h)\,dh
  +
  o(\varepsilon^2),
\end{equation}
where
\[
  \talpha=\frac{p-2}{p+n+1},
  \qquad
  \tbeta=\frac{n+3}{p+n+1}.
\]
The key point is that the lifted test function
\[
  \Phi_x(h,t)=\phi(x+h)+t
\]
has gradient
\[
  D\Phi_x(0,0)=(D\phi(x),1),
\]
which never vanishes. Therefore the extremal expansion needs no separate
case at points where $D\phi(x)=0$.

The projected operator originates in the projected tug-of-war construction
studied in \cite{Moosavi26}. That work develops the discrete foundations
of the scheme and proves convergence of the projected dynamic programming
solutions. The present paper proves the local analytic characterization.

\medskip
\noindent\textbf{Organization.}
\Cref{sec:char} introduces the normalized operator, the projected
ingredients, and the main theorem. \Cref{sec:proof} proves the averaging
and extremal expansions and combines them. \Cref{sec:endpoint} gives the
endpoint example. \Cref{sec:concl} gives concluding remarks.

\medskip
\noindent\textbf{Summary of results.}
\begin{itemize}[leftmargin=1.6em]
\item The main result (\cref{thm:main}) is a pointwise asymptotic mean
value characterization of \eqref{eq:reg-pLap}, valid for every
$1<p<\infty$.
\item The identity comes from the affine lift $w(x,s)=v(x)+s$. Because the
lifted gradient $(Dv,1)$ never vanishes, the extremal expansion is the same
at critical and noncritical points, and no case distinction in $p$ is
needed.
\item The two ingredients of the operator are described explicitly: the
semicircular kernel $\rhoeps$ and the tilted functionals $\Sp,\Sm$. The
averaging term is handled by a projected Pizzetti expansion
(\cref{obs:pizzetti}), of which only the second-order term enters the
proof.
\item \Cref{ex:aronsson} records an endpoint example: the formal
$p\to\infty$ identity does not hold pointwise for a nonsmooth
infinity-harmonic function.
\end{itemize}
\section{The regularized operator and the main characterization}
\label{sec:char}

\subsection{The normalized equation}
\label{subsec:operator}

For $v\in C^2(\Omega)$ and $1<p<\infty$, set
\[
  \Reg_p(v)
  :=
  \operatorname{div}
  \bigl((1+|Dv|^2)^{p/2-1}Dv\bigr).
\]
A direct computation gives
\begin{equation}\label{eq:decomp}
  \Reg_p(v)
  =
  (1+|Dv|^2)^{p/2-1}
  \bigl(\Delta v+(p-2)\DeltaiR v\bigr),
\end{equation}
where
\[
  \DeltaiR v
  :=
  \frac{\langle D^2v\,Dv,Dv\rangle}{1+|Dv|^2}.
\]
Since
\[
  (1+|Dv|^2)^{p/2-1}>0,
\]
we have
\[
  \Reg_p(v)=0
  \quad\Longleftrightarrow\quad
  \Delta v+(p-2)\DeltaiR v=0.
\]

Define
\begin{equation}\label{eq:L-operator}
  \Lp\phi
  :=
  \Delta\phi
  +(p-2)
  \frac{\langle D^2\phi\,D\phi,D\phi\rangle}{1+|D\phi|^2}.
\end{equation}
Then
\[
  \Reg_p(\phi)
  =
  (1+|D\phi|^2)^{p/2-1}\Lp\phi .
\]
In coordinates,
\[
  \Lp\phi
  =
  a_{ij}(D\phi)\phi_{ij},
\]
where
\[
  a_{ij}(\xi)
  :=
  \delta_{ij}
  +(p-2)\frac{\xi_i\xi_j}{1+|\xi|^2}.
\]
For every $\eta\in\R^n$,
\[
  a_{ij}(\xi)\eta_i\eta_j
  =
  |\eta|^2
  +
  (p-2)\frac{\langle \xi,\eta\rangle^2}{1+|\xi|^2}.
\]
Hence
\[
  \min(1,p-1)|\eta|^2
  \le
  a_{ij}(\xi)\eta_i\eta_j
  \le
  \max(1,p-1)|\eta|^2.
\]
Thus $\Lp$ is uniformly elliptic for every finite $1<p<\infty$.

We shall use the standard fact that continuous weak solutions of the
regularized equation are classical. More precisely, if
\[
  v\in W^{1,p}_{\mathrm{loc}}(\Omega)
\]
satisfies
\begin{equation}\label{eq:weak-form}
  \int_\Omega
  (1+|Dv|^2)^{(p-2)/2}
  \langle Dv,D\varphi\rangle\,dx
  =
  0
  \qquad
  \text{for every }\varphi\in C_c^\infty(\Omega),
\end{equation}
then
\[
  v\in C^\infty(\Omega),
  \qquad
  \Reg_p(v)=0
  \quad\text{classically in }\Omega .
\]
See \cite{Lieberman88,GilbargTrudinger}. For the weak--viscosity
equivalence in the related degenerate $p$-Laplace setting, see
\cite{JLM01}.

\subsection{The projected ingredients}
\label{subsec:scheme}

Let $\varepsilon>0$. Define
\begin{equation}\label{eq:weights}
  \talpha=\frac{p-2}{p+n+1},
  \qquad
  \tbeta=\frac{n+3}{p+n+1}.
\end{equation}
Then
\[
  \talpha+\tbeta=1.
\]

The projected averaging kernel is
\begin{equation}\label{eq:rho}
  \rhoeps(h)
  :=
  \frac{2\sqrt{\varepsilon^2-|h|^2}}{|B_\varepsilon^{n+1}|},
  \qquad
  h\in B_\varepsilon(0)\subset\R^n.
\end{equation}
It is the marginal density of normalized Lebesgue measure on
$B_\varepsilon^{n+1}(0)\subset\R^{n+1}$.

For a function $u$ defined on a neighborhood of
$\overline{B_\varepsilon(x)}$, define
\begin{align}
  \Sp[u](x)
  &:=
  \sup_{|h|\le\varepsilon}
  \left(
    u(x+h)+\sqrt{\varepsilon^2-|h|^2}
  \right),
  \label{eq:Splus}\\
  \Sm[u](x)
  &:=
  \inf_{|h|\le\varepsilon}
  \left(
    u(x+h)-\sqrt{\varepsilon^2-|h|^2}
  \right).
  \label{eq:Sminus}
\end{align}
The projected operator is
\begin{equation}\label{eq:proj-DPP}
  T_\varepsilon u(x)
  :=
  \frac{\talpha}{2}
  \bigl(\Sp[u](x)+\Sm[u](x)\bigr)
  +
  \tbeta
  \int_{B_\varepsilon(0)}
  u(x+h)\rhoeps(h)\,dh.
\end{equation}

These quantities come from the affine lift
\[
  w(x,s)=u(x)+s.
\]
For fixed $x$, set
\[
  \Phi_x(h,t):=u(x+h)+t.
\]
Then
\[
  \Sp[u](x)
  =
  \max_{\overline{B_\varepsilon^{n+1}(0)}}\Phi_x,
  \qquad
  \Sm[u](x)
  =
  \min_{\overline{B_\varepsilon^{n+1}(0)}}\Phi_x.
\]
Moreover, the average of $\Phi_x$ over $B_\varepsilon^{n+1}(0)$ projects
to the $\rhoeps$-average of $u$ in the $h$-variables.

\subsection{The main theorem}
\label{subsec:main-theorem}

\begin{theorem}[Pointwise asymptotic characterization]
\label{thm:main}
Let $\Omega\subset\R^n$ be open, let $1<p<\infty$, and let
$v\in C^2(\Omega)$. The following are equivalent:
\begin{enumerate}[label=\textup{(\roman*)}]
\item $\Reg_p(v)=0$ in $\Omega$;
\item for every $x\in\Omega$,
\begin{equation}\label{eq:amv}
  v(x)
  =
  \frac{\talpha}{2}
  \bigl(\Sp[v](x)+\Sm[v](x)\bigr)
  +
  \tbeta
  \int_{B_\varepsilon(0)}
  v(x+h)\rhoeps(h)\,dh
  +
  o(\varepsilon^2)
  \qquad\text{as }\varepsilon\to0.
\end{equation}
\end{enumerate}
Equivalently, for every $x\in\Omega$,
\begin{equation}\label{eq:main-consistency}
  T_\varepsilon v(x)-v(x)
  =
  \frac{\varepsilon^2}{2(p+n+1)}
  \Lp v(x)
  +
  o(\varepsilon^2)
  \qquad\text{as }\varepsilon\to0.
\end{equation}
The same conclusion holds for continuous weak solutions of
\eqref{eq:weak-form}, by the regularity statement above.
\end{theorem}

For $p=2$ one has
\[
  \talpha=0,
  \qquad
  \tbeta=1.
\]
Thus \eqref{eq:amv} reduces to
\[
  v(x)
  =
  \int_{B_\varepsilon(0)}
  v(x+h)\rhoeps(h)\,dh
  +
  o(\varepsilon^2).
\]
The second-order coefficient of this projected average is
$1/(2(n+3))$, reflecting the fact that the average is the solid mean in
dimension $n+1$ applied to a lifted function independent of the last
variable.

Formally, as $p\to\infty$,
\[
  \talpha\to1,
  \qquad
  \tbeta\to0,
\]
and the projected identity becomes
\begin{equation}\label{eq:pinf-charact}
  v(x)
  =
  \frac12
  \bigl(\Sp[v](x)+\Sm[v](x)\bigr)
  +
  o(\varepsilon^2).
\end{equation}
For smooth functions this corresponds to
\[
  \DeltaiR v
  =
  \frac{\langle D^2v\,Dv,Dv\rangle}{1+|Dv|^2}
  =
  0.
\]

\section{Proof of the characterization theorem}
\label{sec:proof}

The proof combines two local expansions: a Pizzetti-type expansion for
the projected average, and a max--min expansion for the tilted extremal
terms.

\subsection{The averaging expansion}
\label{subsec:p2}

For $\phi\in C(U)$ define
\[
  A_\varepsilon\phi(x)
  :=
  \int_{B_\varepsilon(0)}
  \phi(x+h)\rhoeps(h)\,dh .
\]

\begin{lemma}[The projected average is a lifted ball mean]
\label{lem:avg-lift}
Let $U\subset\R^n$ be open, let $\phi\in C(U)$, and let
$x\in U$ with $\overline{B_\varepsilon(x)}\subset U$. Then
\[
  A_\varepsilon\phi(x)
  =
  \fint_{B_\varepsilon^{n+1}(0)}
  \phi(x+h)\,d(h,t).
\]
\end{lemma}

\begin{proof}
For fixed $h\in B_\varepsilon(0)$,
\[
  \{t:(h,t)\in B_\varepsilon^{n+1}(0)\}
  =
  \left[
    -\sqrt{\varepsilon^2-|h|^2},
    \sqrt{\varepsilon^2-|h|^2}
  \right].
\]
Thus
\[
  \fint_{B_\varepsilon^{n+1}(0)}
  \phi(x+h)\,d(h,t)
  =
  \frac{1}{|B_\varepsilon^{n+1}|}
  \int_{B_\varepsilon(0)}
  \phi(x+h)\,2\sqrt{\varepsilon^2-|h|^2}\,dh.
\]
By the definition of $\rhoeps$, this is $A_\varepsilon\phi(x)$.
\end{proof}

\begin{observation}[Projected Pizzetti expansion]
\label{obs:pizzetti}
Let $U\subset\R^n$ be open, let $k\ge1$ be an integer, let
$\phi\in C^{2k}(U)$, and let $x\in U$ with
$\overline{B_\varepsilon(x)}\subset U$. Then
\begin{equation}\label{eq:pizzetti}
  A_\varepsilon\phi(x)
  =
  \sum_{j=0}^{k}
  c_{j,n}\varepsilon^{2j}\Delta^j\phi(x)
  +
  o(\varepsilon^{2k}),
\end{equation}
where
\[
  c_{0,n}=1,
\]
and, for $1\le j\le k$,
\begin{equation}\label{eq:pizzetti-coeff}
  c_{j,n}
  =
  \frac{1}
  {2^j j!(n+3)(n+5)\cdots(n+2j+1)}
  =
  \frac{1}{4^j j!}
  \frac{\Gamma\!\left(\frac{n+3}{2}\right)}
       {\Gamma\!\left(j+\frac{n+3}{2}\right)}.
\end{equation}
Only even powers of $\varepsilon$ occur.
\end{observation}

\begin{proof}
Taylor expansion gives
\[
  \phi(x+h)
  =
  \sum_{m=0}^{2k}P_m(h)+r(x,h),
\]
where
\[
  P_m(h)
  :=
  \frac1{m!}
  \sum_{i_1,\dots,i_m=1}^{n}
  \phi_{i_1\cdots i_m}(x)
  h_{i_1}\cdots h_{i_m},
\]
and
\[
  r(x,h)=o(|h|^{2k}).
\]
Integrating against $\rhoeps$ gives
\[
  A_\varepsilon\phi(x)
  =
  \sum_{m=0}^{2k}
  \int_{B_\varepsilon(0)}P_m(h)\rhoeps(h)\,dh
  +
  \int_{B_\varepsilon(0)}r(x,h)\rhoeps(h)\,dh.
\]
If $m$ is odd, then every monomial of degree $m$ has at least one
coordinate with odd exponent. Since $B_\varepsilon(0)$ is symmetric and
$\rhoeps$ is radial,
\[
  \int_{B_\varepsilon(0)}P_m(h)\rhoeps(h)\,dh=0.
\]
Thus
\[
  A_\varepsilon\phi(x)
  =
  \sum_{j=0}^{k}
  \int_{B_\varepsilon(0)}P_{2j}(h)\rhoeps(h)\,dh
  +
  \int_{B_\varepsilon(0)}r(x,h)\rhoeps(h)\,dh.
\]

For a homogeneous polynomial $P$ of degree $2j$, radial averaging gives
\begin{equation}\label{eq:radial-average}
  \int_{B_\varepsilon(0)}P(h)\rhoeps(h)\,dh
  =
  \frac{\Delta^jP}{\Delta^j|h|^{2j}}
  \int_{B_\varepsilon(0)}|h|^{2j}\rhoeps(h)\,dh.
\end{equation}
Indeed, in the Fischer decomposition
\[
  P(h)=\sum_{i=0}^{j}|h|^{2i}H_{2j-2i}(h),
\]
where $H_{2j-2i}$ is harmonic and homogeneous of degree $2j-2i$, all
terms with positive harmonic degree have zero spherical mean. Hence only
the radial term survives under radial integration. Applying $\Delta^j$
also isolates the same radial term.

Applying \eqref{eq:radial-average} to $P=P_{2j}$ gives
\[
  \int_{B_\varepsilon(0)}P_{2j}(h)\rhoeps(h)\,dh
  =
  \Delta^j\phi(x)
  \frac{
    \displaystyle
    \int_{B_\varepsilon(0)}|h|^{2j}\rhoeps(h)\,dh
  }{
    \displaystyle
    \Delta^j|h|^{2j}
  }.
\]
Now
\[
  \Delta |h|^{2j}
  =
  2j(2j+n-2)|h|^{2j-2}.
\]
Repeating this identity,
\begin{equation}\label{eq:lap-power}
  \Delta^j|h|^{2j}
  =
  4^j j!
  \frac{\Gamma\!\left(j+\frac n2\right)}
       {\Gamma\!\left(\frac n2\right)}.
\end{equation}
Also,
\[
  \int_{B_\varepsilon(0)}|h|^{2j}\rhoeps(h)\,dh
  =
  \frac{2}{|B_\varepsilon^{n+1}|}
  \int_{B_\varepsilon(0)}
  |h|^{2j}\sqrt{\varepsilon^2-|h|^2}\,dh.
\]
Radial integration in $\R^n$ and the beta integral give
\begin{equation}\label{eq:moment-power}
  \int_{B_\varepsilon(0)}|h|^{2j}\rhoeps(h)\,dh
  =
  \varepsilon^{2j}
  \frac{
    \Gamma\!\left(\frac{n+3}{2}\right)
    \Gamma\!\left(j+\frac n2\right)
  }{
    \Gamma\!\left(\frac n2\right)
    \Gamma\!\left(j+\frac{n+3}{2}\right)
  }.
\end{equation}
Dividing \eqref{eq:moment-power} by \eqref{eq:lap-power} gives
\[
  \int_{B_\varepsilon(0)}P_{2j}(h)\rhoeps(h)\,dh
  =
  \varepsilon^{2j}
  \frac{1}{4^j j!}
  \frac{\Gamma\!\left(\frac{n+3}{2}\right)}
       {\Gamma\!\left(j+\frac{n+3}{2}\right)}
  \Delta^j\phi(x).
\]
This is the Gamma form of \eqref{eq:pizzetti-coeff}. The product form
follows from
\[
  \Gamma\!\left(j+\frac{n+3}{2}\right)
  =
  \Gamma\!\left(\frac{n+3}{2}\right)
  \prod_{q=0}^{j-1}
  \left(\frac{n+3}{2}+q\right).
\]
The remainder is $o(\varepsilon^{2k})$ because
\[
  r(x,h)=o(|h|^{2k})
\]
and
\[
  \int_{B_\varepsilon(0)}|h|^{2k}\rhoeps(h)\,dh
  =
  O(\varepsilon^{2k}).
\]
This proves \eqref{eq:pizzetti}.
\end{proof}

\begin{remark}
By \cref{lem:avg-lift}, $A_\varepsilon\phi(x)$ is the solid ball mean in
dimension $n+1$ of the function $(h,t)\mapsto\phi(x+h)$, which is
independent of $t$. Thus \cref{obs:pizzetti} is the classical solid
Pizzetti expansion in dimension $n+1$, written after projection onto the
$h$-variables.
\end{remark}

\begin{proposition}[Second-order averaging expansion]
\label{prop:p2}
Let $U\subset\R^n$ be open, let $\phi\in C^2(U)$, and let
$x\in U$ with $\overline{B_\varepsilon(x)}\subset U$. Then
\begin{equation}\label{eq:p2-expand}
  \int_{B_\varepsilon(0)}
  \phi(x+h)\rhoeps(h)\,dh
  =
  \phi(x)
  +
  \frac{\varepsilon^2}{2(n+3)}
  \Delta\phi(x)
  +
  o(\varepsilon^2).
\end{equation}
\end{proposition}

\begin{proof}
Take $k=1$ in \cref{obs:pizzetti}. Then
\[
  c_{0,n}=1,
  \qquad
  c_{1,n}
  =
  \frac{1}{2(n+3)}.
\]
\end{proof}

\subsection{The tilted extremal expansion}
\label{subsec:pinf}

For $\phi\in C^2(U)$ define
\begin{equation}\label{eq:Phi-def}
  \Phi_x(h,t):=\phi(x+h)+t,
  \qquad
  (h,t)\in\R^n\times\R.
\end{equation}

\begin{lemma}[Lifted representation]
\label{lem:lift-id}
If $\overline{B_\varepsilon(x)}\subset U$, then
\[
  \Sp[\phi](x)
  =
  \max_{\overline{B_\varepsilon^{n+1}(0)}}\Phi_x,
  \qquad
  \Sm[\phi](x)
  =
  \min_{\overline{B_\varepsilon^{n+1}(0)}}\Phi_x.
\]
\end{lemma}

\begin{proof}
Fix $h\in\overline{B_\varepsilon(0)}$. The condition
\[
  |h|^2+t^2\le\varepsilon^2
\]
is equivalent to
\[
  -\sqrt{\varepsilon^2-|h|^2}
  \le
  t
  \le
  \sqrt{\varepsilon^2-|h|^2}.
\]
Since
\[
  \Phi_x(h,t)=\phi(x+h)+t
\]
is increasing in $t$, the maximum is attained at
$t=\sqrt{\varepsilon^2-|h|^2}$ and the minimum at
$t=-\sqrt{\varepsilon^2-|h|^2}$. This gives the two identities.
\end{proof}

\begin{lemma}[Max--min expansion]
\label{lem:max-min}
Let $g\in C^2$ in a neighborhood of $0\in\R^m$, with $Dg(0)\ne0$.
Then, as $\varepsilon\to0$,
\begin{align}
  \max_{\overline{B_\varepsilon(0)}}g
  &=
  g(0)
  +
  \varepsilon|Dg(0)|
  +
  \frac{\varepsilon^2}{2}
  \frac{
    \langle D^2g(0)Dg(0),Dg(0)\rangle
  }{
    |Dg(0)|^2
  }
  +
  o(\varepsilon^2),
  \label{eq:maxg}\\
  \min_{\overline{B_\varepsilon(0)}}g
  &=
  g(0)
  -
  \varepsilon|Dg(0)|
  +
  \frac{\varepsilon^2}{2}
  \frac{
    \langle D^2g(0)Dg(0),Dg(0)\rangle
  }{
    |Dg(0)|^2
  }
  +
  o(\varepsilon^2).
  \label{eq:ming}
\end{align}
Consequently,
\begin{equation}\label{eq:max-min-avg}
  \frac12
  \left(
    \max_{\overline{B_\varepsilon(0)}}g
    +
    \min_{\overline{B_\varepsilon(0)}}g
  \right)
  =
  g(0)
  +
  \frac{\varepsilon^2}{2}
  \frac{
    \langle D^2g(0)Dg(0),Dg(0)\rangle
  }{
    |Dg(0)|^2
  }
  +
  o(\varepsilon^2).
\end{equation}
\end{lemma}

\begin{proof}
Set
\[
  q:=Dg(0),
  \qquad
  e:=\frac{q}{|q|}.
\]
Taylor expansion gives, uniformly for $|z|\le\varepsilon$,
\begin{equation}\label{eq:Taylor-g}
  g(z)
  =
  g(0)
  +
  q\cdot z
  +
  \frac12\langle D^2g(0)z,z\rangle
  +
  o(\varepsilon^2).
\end{equation}
For small $\varepsilon$, the maximizer lies on $\partial B_\varepsilon(0)$.
Write
\[
  z_\varepsilon=\varepsilon\omega_\varepsilon,
  \qquad
  |\omega_\varepsilon|=1.
\]
Comparison with $\varepsilon e$ gives
\[
  q\cdot\omega_\varepsilon
  \ge
  |q|+O(\varepsilon).
\]
Since $q\cdot\omega_\varepsilon\le |q|$, we have
\[
  \omega_\varepsilon\to e.
\]
The constrained critical point equation on the sphere gives
\[
  \omega_\varepsilon=e+O(\varepsilon).
\]
Therefore
\[
  q\cdot z_\varepsilon
  =
  \varepsilon |q|+O(\varepsilon^3),
\]
and
\[
  \frac12\langle D^2g(0)z_\varepsilon,z_\varepsilon\rangle
  =
  \frac{\varepsilon^2}{2}
  \langle D^2g(0)e,e\rangle
  +
  o(\varepsilon^2).
\]
Substituting into \eqref{eq:Taylor-g} gives \eqref{eq:maxg}. Applying the
same argument to $-g$ gives \eqref{eq:ming}. Averaging gives
\eqref{eq:max-min-avg}.
\end{proof}

\begin{lemma}[Quantitative max--min expansion]
\label{lem:max-min-quantitative}
Let $g\in C^3$ in a neighborhood of $0\in\R^m$ with $q:=Dg(0)\ne0$, and
set
\[
  e:=\frac{q}{|q|}.
\]
Then there exist $\rho_0>0$ and $C>0$ such that, for
$0<\rho\le\rho_0$, the maximizer $z^+(\rho)$ and minimizer $z^-(\rho)$
of $g$ over $\overline{B_\rho(0)}$ are unique, lie on
$\partial B_\rho(0)$, and satisfy
\begin{equation}\label{eq:zpm-expansion}
  z^\pm(\rho)=\pm\rho e+O(\rho^2).
\end{equation}
Moreover, if
\[
  M^+(\rho):=
  \max_{\overline{B_\rho(0)}}g,
  \qquad
  M^-(\rho):=
  \min_{\overline{B_\rho(0)}}g,
\]
then
\[
  M^\pm\in C^1((0,\rho_0))
\]
and
\begin{equation}\label{eq:envelope-identity}
  (M^\pm)'(\rho)
  =
  \pm
  |Dg(z^\pm(\rho))|.
\end{equation}
Consequently, writing
\[
  R^\pm(\rho)
  :=
  M^\pm(\rho)-g(0)\mp\rho|q|
  -
  \frac{\rho^2}{2}
  \frac{\langle D^2g(0)q,q\rangle}{|q|^2},
\]
one has
\begin{equation}\label{eq:quantitative-remainder}
  |R^\pm(\rho)|\le C\rho^3,
  \qquad
  |(R^\pm)'(\rho)|\le C\rho^2.
\end{equation}
\end{lemma}

\begin{proof}
Since $q\ne0$, after decreasing $\rho_0$ we have $Dg\ne0$ on
$\overline{B_{\rho_0}(0)}$. Hence extrema over
$\overline{B_\rho(0)}$ occur on $\partial B_\rho(0)$.

Write $z=\rho\omega$, $|\omega|=1$. The expansion
\[
  g(\rho\omega)
  =
  g(0)+\rho q\cdot\omega+O(\rho^2)
\]
shows that maximizing directions converge to $e$, and minimizing
directions converge to $-e$.

The constrained critical point equation is
\[
  Dg(\rho\omega)=\lambda\omega,
  \qquad
  |\omega|=1.
\]
The implicit function theorem applied at
\[
  (e,|q|,0)
  \quad\text{and}\quad
  (-e,-|q|,0)
\]
gives unique $C^2$ branches
\[
  \omega^\pm(\rho)=\pm e+O(\rho).
\]
Thus
\[
  z^\pm(\rho)=\rho\omega^\pm(\rho)=\pm\rho e+O(\rho^2).
\]

Since $|\omega^\pm(\rho)|=1$,
\[
  \omega^\pm(\rho)\cdot(\omega^\pm)'(\rho)=0.
\]
Using the constrained critical point equation,
\[
  (M^\pm)'(\rho)
  =
  Dg(\rho\omega^\pm)
  \cdot
  \left(
    \omega^\pm+\rho(\omega^\pm)'
  \right)
  =
  \lambda^\pm.
\]
The multiplier is positive at the maximum and negative at the minimum for
small $\rho$, so
\[
  (M^\pm)'(\rho)
  =
  \pm|Dg(z^\pm(\rho))|.
\]

Finally,
\[
  Dg(z^\pm(\rho))
  =
  q\pm\rho D^2g(0)e+O(\rho^2).
\]
Therefore
\[
  (M^\pm)'(\rho)
  =
  \pm |q|
  +
  \rho
  \frac{\langle D^2g(0)q,q\rangle}{|q|^2}
  +
  O(\rho^2).
\]
This gives
\[
  |(R^\pm)'(\rho)|\le C\rho^2.
\]
Integrating from $0$ to $\rho$ gives
\[
  |R^\pm(\rho)|\le C\rho^3.
\]
\end{proof}

\begin{proposition}[Tilted extremal expansion]
\label{prop:pinf}
Let $U\subset\R^n$ be open, let $\phi\in C^2(U)$, and let
$x\in U$ with $\overline{B_\varepsilon(x)}\subset U$. Then
\begin{equation}\label{eq:pinf-expand}
  \frac12
  \bigl(\Sp[\phi](x)+\Sm[\phi](x)\bigr)
  =
  \phi(x)
  +
  \frac{\varepsilon^2}{2}
  \DeltaiR\phi(x)
  +
  o(\varepsilon^2).
\end{equation}
\end{proposition}

\begin{proof}
By \cref{lem:lift-id},
\[
  \Sp[\phi](x)
  =
  \max_{\overline{B_\varepsilon^{n+1}(0)}}\Phi_x,
  \qquad
  \Sm[\phi](x)
  =
  \min_{\overline{B_\varepsilon^{n+1}(0)}}\Phi_x.
\]
At the origin,
\[
  D\Phi_x(0,0)=(D\phi(x),1),
\]
and
\[
  D^2\Phi_x(0,0)
  =
  \begin{pmatrix}
    D^2\phi(x) & 0\\
    0 & 0
  \end{pmatrix}.
\]
Hence
\[
  |D\Phi_x(0,0)|^2
  =
  1+|D\phi(x)|^2,
\]
and
\[
  \langle
    D^2\Phi_x(0,0)D\Phi_x(0,0),
    D\Phi_x(0,0)
  \rangle
  =
  \langle
    D^2\phi(x)D\phi(x),
    D\phi(x)
  \rangle.
\]
Since $D\Phi_x(0,0)\ne0$, \cref{lem:max-min} applies to $g=\Phi_x$.
Thus
\[
  \frac12
  \bigl(\Sp[\phi](x)+\Sm[\phi](x)\bigr)
  =
  \phi(x)
  +
  \frac{\varepsilon^2}{2}
  \frac{
    \langle D^2\phi(x)D\phi(x),D\phi(x)\rangle
  }{
    1+|D\phi(x)|^2
  }
  +
  o(\varepsilon^2),
\]
which is \eqref{eq:pinf-expand}.
\end{proof}

\subsection{The combined expansion}
\label{subsec:general}

\begin{lemma}[Combined expansion]
\label{lem:combined}
Let $U\subset\R^n$ be open, let $\phi\in C^2(U)$, and let
$x\in U$ with $\overline{B_\varepsilon(x)}\subset U$. Then
\begin{equation}\label{eq:combined}
  T_\varepsilon\phi(x)-\phi(x)
  =
  \frac{\varepsilon^2}{2(p+n+1)}
  \Lp\phi(x)
  +
  o(\varepsilon^2).
\end{equation}
\end{lemma}

\begin{proof}
By \cref{prop:pinf},
\[
  \frac12
  \bigl(\Sp[\phi](x)+\Sm[\phi](x)\bigr)
  =
  \phi(x)
  +
  \frac{\varepsilon^2}{2}
  \DeltaiR\phi(x)
  +
  o(\varepsilon^2).
\]
By \cref{prop:p2},
\[
  \int_{B_\varepsilon(0)}
  \phi(x+h)\rhoeps(h)\,dh
  =
  \phi(x)
  +
  \frac{\varepsilon^2}{2(n+3)}
  \Delta\phi(x)
  +
  o(\varepsilon^2).
\]
Substituting into \eqref{eq:proj-DPP},
\[
  T_\varepsilon\phi(x)
  =
  \talpha
  \left(
    \phi(x)
    +
    \frac{\varepsilon^2}{2}
    \DeltaiR\phi(x)
  \right)
  +
  \tbeta
  \left(
    \phi(x)
    +
    \frac{\varepsilon^2}{2(n+3)}
    \Delta\phi(x)
  \right)
  +
  o(\varepsilon^2).
\]
Since $\talpha+\tbeta=1$,
\[
  T_\varepsilon\phi(x)-\phi(x)
  =
  \frac{\varepsilon^2}{2}
  \left(
    \talpha\DeltaiR\phi(x)
    +
    \frac{\tbeta}{n+3}\Delta\phi(x)
  \right)
  +
  o(\varepsilon^2).
\]
Using
\[
  \talpha=\frac{p-2}{p+n+1},
  \qquad
  \frac{\tbeta}{n+3}
  =
  \frac{1}{p+n+1},
\]
we get
\[
  T_\varepsilon\phi(x)-\phi(x)
  =
  \frac{\varepsilon^2}{2(p+n+1)}
  \left(
    \Delta\phi(x)
    +
    (p-2)\DeltaiR\phi(x)
  \right)
  +
  o(\varepsilon^2).
\]
This is \eqref{eq:combined}.
\end{proof}

\begin{corollary}[Locally uniform consistency]
\label{cor:uniform-consistency}
Let $U\subset\R^n$ be open and let $K\Subset U$. Then, for every
$\phi\in C^2(U)$,
\[
  T_\varepsilon\phi(x)-\phi(x)
  =
  \frac{\varepsilon^2}{2(p+n+1)}
  \Lp\phi(x)
  +
  o(\varepsilon^2)
\]
uniformly for $x\in K$.
\end{corollary}

\begin{proof}
The averaging remainder is locally uniform by the continuity of
$D^2\phi$ on compact subsets of $U$. The tilted extremal remainder is
locally uniform because
\[
  D\Phi_x(0,0)=(D\phi(x),1)
\]
has norm bounded below by $1$ uniformly in $x$.
\end{proof}

\begin{proof}[Proof of \cref{thm:main}]
Let $v\in C^2(\Omega)$ and fix $x\in\Omega$. Choose $\varepsilon>0$ small
enough that
\[
  \overline{B_\varepsilon(x)}\subset\Omega.
\]
By \cref{lem:combined},
\begin{equation}\label{eq:proof-main-combined}
  T_\varepsilon v(x)-v(x)
  =
  \frac{\varepsilon^2}{2(p+n+1)}
  \Lp v(x)
  +
  o(\varepsilon^2).
\end{equation}

If $\Reg_p(v)=0$, then by \eqref{eq:decomp},
\[
  (1+|Dv|^2)^{p/2-1}\Lp v=0.
\]
Since the prefactor is positive, $\Lp v=0$. Hence
\[
  T_\varepsilon v(x)-v(x)=o(\varepsilon^2),
\]
which is \eqref{eq:amv}.

Conversely, suppose \eqref{eq:amv} holds. Then
\[
  T_\varepsilon v(x)-v(x)=o(\varepsilon^2).
\]
Combining this with \eqref{eq:proof-main-combined},
\[
  o(\varepsilon^2)
  =
  \frac{\varepsilon^2}{2(p+n+1)}
  \Lp v(x)
  +
  o(\varepsilon^2).
\]
Dividing by $\varepsilon^2$ and letting $\varepsilon\to0$ gives
\[
  \Lp v(x)=0.
\]
Since $x$ was arbitrary, $\Lp v=0$ in $\Omega$. By \eqref{eq:decomp},
\[
  \Reg_p(v)=0.
\]
The weak-solution statement follows from the regularity recalled in
\cref{subsec:operator}.
\end{proof}

\section{A pointwise endpoint example}
\label{sec:endpoint}

The following example shows that the formal endpoint identity
\eqref{eq:pinf-charact} does not extend as a pointwise identity to a
standard nonsmooth infinity-harmonic function.

\begin{example}[The projected endpoint fails pointwise]
\label{ex:aronsson}
Consider Aronsson's function
\begin{equation}\label{eq:aronsson}
  u(x,y)=|x|^{4/3}-|y|^{4/3}.
\end{equation}
This function is infinity harmonic in the viscosity sense
\cite{Aronsson84}, but it is not $C^2$ at $(1,0)$, where
\[
  Du(1,0)=\left(\frac43,0\right)\ne0.
\]
We show that \eqref{eq:pinf-charact} fails at $(1,0)$.

By \cref{lem:lift-id}, $\Sp[u](1,0)$ and $\Sm[u](1,0)$ are the maximum
and minimum, over
\[
  h_1^2+h_2^2+t^2\le\varepsilon^2,
\]
of
\[
  \Phi(h_1,h_2,t)
  =
  (1+h_1)^{4/3}-|h_2|^{4/3}+t.
\]
Near the origin,
\[
  (1+h_1)^{4/3}
  =
  1+\frac43h_1+\frac29h_1^2+O(h_1^3).
\]

For the supremum, the term $-|h_2|^{4/3}$ is maximized at $h_2=0$.
Thus we maximize
\[
  g(h_1,t)=(1+h_1)^{4/3}+t
\]
over
\[
  h_1^2+t^2\le\varepsilon^2.
\]
At the origin,
\[
  Dg(0)=\left(\frac43,1\right),
  \qquad
  |Dg(0)|=\frac53,
\]
and
\[
  g_{h_1h_1}(0)=\frac49.
\]
Hence
\[
  \frac{\langle D^2g(0)Dg(0),Dg(0)\rangle}{|Dg(0)|^2}
  =
  \frac{\frac49\left(\frac43\right)^2}{\frac{25}{9}}
  =
  \frac{64}{225}.
\]
By \cref{lem:max-min},
\[
  \Sp[u](1,0)
  =
  1+\frac53\varepsilon+\frac{32}{225}\varepsilon^2
  +
  o(\varepsilon^2).
\]

For the infimum, write $s:=|h_2|$ and
\[
  r=\sqrt{\varepsilon^2-s^2}.
\]
For fixed $s$, the remaining minimization in $(h_1,t)$ is over
\[
  h_1^2+t^2\le r^2.
\]
Let
\[
  M(r)
  :=
  \min_{h_1^2+t^2\le r^2}g(h_1,t).
\]
By \cref{lem:max-min-quantitative},
\[
  M(r)
  =
  1-\frac53r+\frac{32}{225}r^2+R(r),
\]
where
\[
  |R(r)|\le Cr^3,
  \qquad
  |R'(r)|\le Cr^2.
\]
Therefore
\[
  \Sm[u](1,0)
  =
  \min_{0\le s\le\varepsilon}
  \left[
    M\!\left(\sqrt{\varepsilon^2-s^2}\right)-s^{4/3}
  \right].
\]
For $s=o(\varepsilon)$,
\[
  \sqrt{\varepsilon^2-s^2}
  =
  \varepsilon-\frac{s^2}{2\varepsilon}
  +
  O\!\left(\frac{s^4}{\varepsilon^3}\right).
\]
Thus
\[
  M\!\left(\sqrt{\varepsilon^2-s^2}\right)-s^{4/3}
  =
  1-\frac53\varepsilon+\frac{32}{225}\varepsilon^2
  +
  \frac{5}{6\varepsilon}s^2
  -
  s^{4/3}
  +
  O\!\left(\frac{s^4}{\varepsilon^3}+s^2+\varepsilon^3\right).
\]
The leading $s$-dependent part is
\[
  H(s)
  =
  \frac{5}{6\varepsilon}s^2-s^{4/3}.
\]
Solving $H'(s)=0$ gives
\[
  s_\varepsilon
  =
  \left(\frac45\varepsilon\right)^{3/2}.
\]
At this point,
\[
  H(s_\varepsilon)
  =
  -\frac{16}{75}\varepsilon^2.
\]
A localization argument using the bound on $R'$ shows that the minimizer
lies at this scale. Hence
\[
  \Sm[u](1,0)
  =
  1-\frac53\varepsilon
  +
  \frac{32}{225}\varepsilon^2
  -
  \frac{16}{75}\varepsilon^2
  +
  o(\varepsilon^2).
\]
Since
\[
  \frac{16}{75}=\frac{48}{225},
\]
we get
\[
  \Sm[u](1,0)
  =
  1-\frac53\varepsilon
  -
  \frac{16}{225}\varepsilon^2
  +
  o(\varepsilon^2).
\]
Therefore
\[
  \frac12
  \bigl(\Sp[u](1,0)+\Sm[u](1,0)\bigr)
  -
  u(1,0)
  =
  \frac{8}{225}\varepsilon^2
  +
  o(\varepsilon^2).
\]
Hence
\[
  \lim_{\varepsilon\to0^+}
  \frac{
    \frac12
    \bigl(\Sp[u](1,0)+\Sm[u](1,0)\bigr)
    -
    u(1,0)
  }{\varepsilon^2}
  =
  \frac{8}{225}
  \ne0.
\]
Thus the pointwise endpoint identity \eqref{eq:pinf-charact} fails at
$(1,0)$.
\end{example}

\section{Concluding remarks}
\label{sec:concl}

\subsection*{An inhomogeneous identity}
The combined expansion \eqref{eq:combined} suggests an inhomogeneous
counterpart. For the normalized equation $\Lp v=f$ with $f$ continuous, the
same computation gives the formal identity
\[
  v(x)=T_\varepsilon v(x)-\frac{\varepsilon^2}{2(p+n+1)}f(x)+o(\varepsilon^2),
\]
and, equivalently, for the divergence-form equation $\Reg_p(v)=f$,
\[
  v(x)=T_\varepsilon v(x)
       -\frac{\varepsilon^2}{2(p+n+1)}(1+|Dv(x)|^2)^{1-p/2}f(x)
       +o(\varepsilon^2),
\]
the gradient-dependent factor $(1+|Dv|^2)^{1-p/2}$ coming from the
decomposition \eqref{eq:decomp}. Reading these identities in reverse, as in
\cref{thm:main}, characterizes classical solutions of the corresponding
inhomogeneous equations through the local operator $T_\varepsilon$.

\subsection*{The range \texorpdfstring{$1<p<2$}{1<p<2}}
The characterization \cref{thm:main} holds on the full range
$1<p<\infty$. The reason is that the tilted extremal expansion
\eqref{eq:pinf-expand} is a two-sided \emph{equality} for smooth $\phi$,
valid at every gradient because the lifted gradient
$D\Phi_x(0,0)=(D\phi(x),1)$ never vanishes; the combined expansion
\eqref{eq:combined} is then an algebraic identity in which the sign of
$p-2$ enters only as the coefficient $\talpha$, and the proof of
\cref{thm:main} nowhere uses monotonicity of $T_\varepsilon$. In the
nonsmooth viscosity setting, where the extremal term is available only as
a one-sided bound, the same conclusion is reached by selecting the
minimum-point expansion when $p\ge2$ and the maximum-point expansion when
$1<p<2$; this is the device of Manfredi, Parviainen, and Rossi
\cite{MPRparabolic}, in which reversing the sign of $\alpha$ reverses the
inequality but the complementary bound reverses with it, so both bounds
persist. In our smooth setting this selection is unnecessary, and the
characterization is genuinely uniform in $p$.

The sign of $p-2$ does, however, distinguish the two ranges once one passes
from the pointwise identity to the \emph{operator} $T_\varepsilon$ as a
scheme. For $p\ge2$ the weight $\talpha=(p-2)/(p+n+1)$ is nonnegative and
$T_\varepsilon$ is monotone, so it is a genuine dynamic programming
operator; this is the range in which convergence of the projected scheme
is established in the companion paper \cite{Moosavi26}. For $1<p<2$ the
weight $\talpha<0$, so $T_\varepsilon$ assigns a negative coefficient to
the extremal term $\tfrac12(\Sp+\Sm)$ and is no longer monotone in the
naive sense, and the discrete comparison principle underlying the game
value does not apply directly. A route around this obstruction is again
supplied by \cite{MPRparabolic}: writing the limiting operator through the
ordered eigenvalues $\lambda_{\max}\bigl((p-2)D^2\phi\bigr)$ and
$\lambda_{\min}\bigl((p-2)D^2\phi\bigr)$---with the factor $p-2$ kept
inside the eigenvalue, so that it swaps the extreme eigenvalues when
$p<2$---gives a formulation of the comparison principle uniform across
$1<p<\infty$, the singular case following from
$(p-1)\lambda_{\min}+\sum_{\lambda_i\ne\lambda_{\min}}\lambda_i\le0$. We do
not pursue the sub-quadratic convergence theory here; it is the subject of
the companion paper \cite{Moosavi26}.

\bibliographystyle{amsplain}
\bibliography{ref_13_may}
\end{document}